\documentclass{amsart}
\usepackage{amsmath, amssymb, mathrsfs, verbatim, multirow, color, enumerate}

\newcommand{\mathsym}[1]{{}}

\newtheorem{theorem}{Theorem}[section]
\newtheorem{lemma}[theorem]{Lemma}

\theoremstyle{definition}

\newtheorem{rem}[theorem]{Remark}

\newtheorem{probl}[theorem]{Problem}

\theoremstyle{remark}

\numberwithin{equation}{section}
\newcommand{\NN}{\mathbb{N}}
\newcommand{\RR}{\mathbb{R}}
\usepackage{hyperref}

\begin{document}

\title[Determinacy of MP for symmetric algebras of a lc space]{On the determinacy of the moment problem for symmetric algebras of a locally convex space}

\author[M. Infusino, S. Kuhlmann, M. Marshall]{ M. Infusino$^{\dagger}$, S. Kuhlmann$^{\dagger}$, \fbox{ M. Marshall$^{*}$}}
\address{$^{*}$ Department of Mathematics and Statistics,\newline \indent
University of Saskatchewan,\newline \indent
Saskatoon, SK. S7N 5E6, Canada}
\address{$^{\dagger}$Fachbereich Mathematik und Statistik,\newline \indent
Universit\"at Konstanz,\newline \indent
78457 Konstanz, Germany.}
\email{maria.infusino@uni-konstanz.de, salma.kuhlmann@uni-konstanz.de}
\keywords{moment problem; uniqueness; determinacy; symmetric algebras; Suslin spaces; nuclear spaces.}
\subjclass[2010]{Primary 44A60}

\begin{abstract} This note aims to show a uniqueness property for the solution (whenever exists) to the moment problem for the symmetric algebra $S(V)$ of a locally convex space $(V, \tau)$. Let $\mu$ be a measure representing a linear functional $L: S(V)\to\RR$. We deduce a sufficient determinacy condition on $L$ provided that the support of $\mu$ is contained in the union of the topological duals of $V$ with respect to countably many of the seminorms in the family inducing~$\tau$. We compare this result with some already known in literature for such a general form of the moment problem and further discuss how some prior knowledge on the support of the representing measure influences its determinacy. 
\end{abstract}

\maketitle
\noindent\footnotesize{\it Murray Marshall passed away on the 1st of May 2015. He was really passionate about the question addressed in this note and was still working on it in the very last days of his life. We lost a wonderful collaborator and a dear friend. We sorely miss him. (M.~Infusino, S.~Kuhlmann)}\normalsize\\

\section*{Introduction}

Given an integer $d\geq 1$, a linear functional $L:\mathbb{R}[x_1,\dots,x_d]\to\RR$ and a closed subset $K$ of $\RR^d$, the classical $d-$dimensional $K-$moment problem asks whether $L$ can be represented as the integral with respect to a non-negative Radon measure $\mu$ whose support is contained in $K$, i.e., $L(f) = \int f d\mu\,$ $\forall f\in \mathbb{R}[x_1,\dots,x_d]$ and $\mu(\RR^d\setminus K)=0$. If such a measure is unique, then the $K-$moment problem is said to be determinate. In this paper we consider the following general formulation of the moment problem.

Let $A$ be a commutative ring with 1 which is an $\mathbb{R}$-algebra. $X(A)$ denotes the character space
of $A$, i.e., the set of all $\RR-$algebra homomorphisms (that send 1 to~1)
$\alpha : A \rightarrow \mathbb{R}$. For $a \in A$, $\hat{a} : X(A)
\rightarrow \mathbb{R}$ is defined by $\hat{a}(\alpha) = \alpha(a)$.
The only $\RR-$algebra homomorphism from $\mathbb{R}$ to itself is the
identity. $X(A)$ is given the weakest topology such that the
functions $\hat{a}$, $a \in A$, are continuous. 

\begin{probl}\label{MP}
Given a Borel subset $K$ of $X(A)$ and a linear functional $L : A \rightarrow \mathbb{R}$, the $K$-{\it moment problem} asks whether there exists a non-negative Radon measure $\mu$ whose support is contained in $K$ such that $L(a) = \int \hat{a} d\mu\,$ $\forall a\in A$. Such a measure is called a $K$-\emph{representing measure} for $L$ on $A$ and, when it is unique the moment problem is said to be $K$-\emph{determinate}.
\end{probl}
We note that the $d-$dimensional moment problem is a special case of Problem~\ref{MP}. Indeed, $\RR-$algebra homomorphisms from $\mathbb{R}[x_1,\dots,x_d]$ to
$\RR$ correspond to point evaluations $f \mapsto f(\alpha)$, $\alpha
\in \RR^d$ and $X(\mathbb{R}[x_1,\dots,x_d])$ is identified (as a
topological space) with~$\mathbb{R}^d$.

There is a huge literature about moment problems belonging to this general set up (see, e.g., \cite{BCR}, \cite{MGHAS}, \cite{MGES}, \cite{GKM}, \cite{GMW},  \cite{IK-probl}, \cite{L}, \cite{Schmu78}) 
and in particular several works have been devoted to the study of the moment problem for linear functionals defined on the symmetric algebra of certain locally convex spaces (see, e.g.,\! \cite[Chapter~5, Section~2]{BK}, \cite{BS}, \cite{Bor-Yng75}, \cite{Heg75}, \cite{I}, \cite{IKR}, \cite[Section 12.5]{Schmu90}). In \cite{GIKM} we deal with the general case of linear functionals defined on the symmetric algebra $S(V)$ of any locally convex space $(V,\tau)$. More precisely, we provide necessary and sufficient conditions for the existence of a solution to the $K$-moment problem for linear functionals on $S(V)$ which are continuous with respect to the finest locally multiplicatively convex topology $\bar{\tau}$ on $S(V)$ extending $\tau$.

This note is a follow up of \cite{GIKM} and focuses on the determinacy of the moment problem for linear functionals on the symmetric algebra of a locally convex space. Our investigation was indeed motivated by the observation that the continuity assumption in \cite{GIKM} is actually a quite strong determinacy condition with respect to some other ones known in literature for such a problem (c.f. \cite[Remark 6.2 (9), (10)]{GIKM}). These conditions are essentially based on quasi-analytic bounds on the starting linear functional provided the knowledge of some further properties on the support of the representing measure. The aim of this note is to better clarify this deep relation between the knowledge of the support and the generality of the determinacy condition.

In Section 1 we derive some general properties of the Borel $\sigma-$algebra on the topological dual of a locally convex space as consequences of the Banach-Alaoglu-Bourbaki theorem. 

In Section 2 we first introduce our set up for the graded symmetric algebra $S(V)$ of a locally convex space $V$ over $\RR$. Note that we consider the general case when the locally convex topology on $V$ is generated by a possibly uncountable family of seminorms $S$. Then we show, in Theorem \ref{at most one-general}, a sufficient determinacy condition on $L:S(V)\to\RR$ provided that the support of its representing measure is contained in the union of countably many topological duals $V_{\rho}'$ of $V$ with $\rho \in S$. 

In Section 3 we recall one of the most used determinacy result for the moment problem for linear functionals on a special class of locally convex spaces and we observe that actually holds for all the locally convex spaces considered in our Theorem~\ref{at most one-general}. Therefore, we compare these two general results (Theorem \ref{at most one-general} and Theorem \ref{uniqueness-Suslin}) both on the level of the generality of the determinacy conditions and on the level of the pre-knowledge of the support. In the end, we also compare these two theorems applied to the special sort of nuclear locally convex spaces considered, e.g.,\! in \cite{{BK}, {BS}, {I}, {IKR}} and for which a classical theorem by Berezanky, Kondratiev and \v Sifrin provides the existence of a solution to the moment problem with very precise support properties.

\section{Borel sets in the topological dual of a locally convex space}

Let $V$ be an $\mathbb{R}$-vector space and let $S$ be a separating family of seminorms on $V$. We denote by $\tau$ the locally convex (l.c.) topology on $V$ defined by $S$, that is the coarsest topology on $V$ such that all the seminorms in $S$ are continuous. As the topology $\tau$ does not change if we close up $S$ under taking the maximum of finitely many of its elements, we can assume w.l.o.g. that $S$ is directed, i.e.,
$$\forall \rho_1, \rho_2\in S,\,\exists\, \rho\in S, C>0\,\text{ s.t. }\,C\rho(v)\geq \max\{\rho_1(v), \rho_2(v)\},\,\forall\,v\in V.$$

Let $V^*$ be the algebraic dual of $V$ and let us denote by $V'$ the topological dual of $V$, i.e.,\! the subspace of $V^*$ consisting of all $\tau$-continuous elements of $V^*$. For $\rho \in S$, denote by $V_{\rho}'$ the subspace of $V^*$ consisting of all $\rho$-continuous elements of $V^*$. Denote by $\rho'$ the norm on $V_{\rho}'$ defined by $$\rho'(x) := \inf \{ C\geq 0 : |x(f)| \leq C\rho(f)\,, \forall f \in V\}.$$ Then $V_{\rho}'$ is a countable increasing union of balls $$B_i(\rho') := \{ x \in V^* : \rho'(x) \le i\}, \ i = 1,2, \dots.$$ By the Banach-Alaoglu-Bourbaki theorem (see, e.g., \cite[Theorem~(4), p. 248]{Koethe}, \cite[Theorem~3.15]{Rud}), each $B_i(\rho')$ is compact in $V'$ endowed with the weak-* topology $\sigma(V', V)$ and therefore compact in $(V^*, \sigma(V^*, V))$. 
Hence, for any $\rho\in S$ we have that $V_{\rho}'$ is a Borel subset of $(V^*, \sigma(V^*, V))$. Furthermore, it is well-known that a linear functional $L : V \rightarrow \mathbb{R}$ is $\tau$-continuous if and only if there exists $\rho \in S$ such that $L$ is $\rho$-continuous (for a proof see, e.g.,\! \cite[Lemma 4.1]{GIKM}). Hence, $V' = \bigcup_{\rho \in S} V_{\rho}'$. \\

If $S$ is countable, then the general observations above give that $V'$ is a countable union of compact sets and so $V'$ is a Borel set in $V^*$. Here, Borel set in $V^*$ means that it belongs to the Borel $\sigma$-algebra generated by  the open sets in the weak-* topology $\sigma(V^*, V)$ on $V^*$.

\begin{lemma}\label{lemma1} Let $V$ be a $\RR-$vector space endowed with the l.c.\! topology defined by a directed separating family $S$ of seminorms on $V$. If $S$ is countable, then the $\sigma$-algebra of Borel sets of $V'$ coincides with the $\sigma$-algebra of subsets of $V'$ generated by all sets $U_{i,\rho}$, $i \in \NN$, $\rho \in S$ such that $U_{i,\rho}$ is an open subset of $B_i(\rho')$ in the weak-* topology.
\end{lemma}

\begin{proof} Denote the two $\sigma$-algebras by $\Sigma_1$ and $\Sigma_2$ respectively. If $U$ is open in $V'$ then $U$ is the union of the sets $U\cap B_i(\rho')$, $i \in \NN$, $\rho \in S$. Since $S$ is countable, this union is countable. Since each $U \cap B_i(\rho')$ is relatively open in $B_i(\rho')$, this proves $U \in \Sigma_2$. This implies $\Sigma_1 \subseteq \Sigma_2$. On the other hand, each relatively open $U_{i,\rho}$ in $B_i(\rho')$ is expressible as $U_{i,\rho} = U\cap B_i(\rho')$ where $U$ is open in $V'$. Since $B_i(\rho')$ is compact (and hence closed in $V'$), by the Banach-Alaoglu-Bourbaki theorem, this implies $U_{i,\rho} \in \Sigma_1$. This proves $\Sigma_2 \subseteq \Sigma_1$.
\end{proof}

Suppose now that $V$ is separable in the sense that there exists a countable dimensional subspace $V_0$ of $V$ which is $\rho$-dense in $V$ for each $\rho \in S$. Each $\rho \in S$ restricts to a seminorm on $V_0$. Thus we can form the corresponding objects $(V_0)^*$, $(V_0)_{\rho}'$, $(V_0)'$, etc. Density implies that the natural restriction $V_{\rho}' \rightarrow (V_0)_{\rho}'$ is a bijection, for each $\rho \in S$.

\begin{lemma} \label{lemma2} Let $V$ be a $\RR-$vector space endowed with the l.c.\! topology defined by a separating family $S$ of seminorms on $V$. If $S$ is countable and $V$ is separable, then the $\sigma$-algebra of Borel sets of $V'$ is canonically identified with the $\sigma$-algebra of Borel subsets of $(V_0)'$, where $V_0$ is a countable dimensional dense subspace of~$V$.
\end{lemma}

\begin{proof} W.l.o.g. we can assume that $S$ is directed. In view of Lemma \ref{lemma1}, it suffices to show, for each $\rho \in S$ and each $i \in \NN$, that the ball $B_i(\rho')$ in $V_{\rho}'$ is homeomorphic in the weak-* topology to the corresponding ball in $(V_0)_{\rho}'$. Fix a countable dense subset $f_1,f_2,\dots$ of $V_0$. By the Sequential Banach-Alaoglu theorem  \cite[Theorem~3.16]{Rud}, we know that the weak-* topology on $B_i(\rho')$ is the topology defined by the metric $$d(x,y) := \sum_{n=1}^{\infty} 2^{-n}\frac{|\langle x-y,f_n\rangle|}{1+|\langle x-y,f_n\rangle|}.$$ Here, $\langle x,f\rangle := x(f)$. Note that $B_i(\rho') = B_1(\frac{1}{i}\rho')$ and $\frac{1}{i}\rho' = (i\rho)'$. A second application of the Sequential Banach-Alaoglu theorem shows that the weak-* topology on the corresponding ball in $(V_0)'$ is the metric topology defined by exactly the same metric. It follows that these two topologies coincide.
\end{proof}

\section{A determinacy condition for the moment problem on $S(V)$}\label{SecDetCond}
Let us briefly recall the basic notations we are going to use in the following. Let $V$ be an $\mathbb{R}$-vector space. We denote by $S(V)$ the symmetric algebra of~$V$, i.e., the tensor algebra $T(V)$ factored by the ideal generated by the elements $v\otimes w -w\otimes v$, $v,w\in V$. If we fix a basis $x_i$, $i\in \Omega$ of $V$, then  $S(V)$ is identified with the polynomial ring $\mathbb{R}[x_i : i\in \Omega]$, i.e., the free $\mathbb{R}$-algebra in commuting variables $x_i$, $i\in \Omega$. The algebra $S(V)$ is a graded algebra. For any integer $k\ge 0$, denote by $S(V)_k$ the $k$-th homogeneous part of $S(V)$, i.e.,\! the image of $k$-th homogeneous part $V^{\otimes k}$ of $T(V)$ under the canonical map $\sum_{i=1}^n f_{i1}\otimes \cdots \otimes f_{ik} \mapsto \sum_{i=1}^n  f_{i1}\cdots f_{ik}.$ Here, $f_{ij} \in V$ for  $i=1,\dots, n$, $j=1,\dots,k$ and $n\ge 1$. Note that $S(V)_0 = \mathbb{R}$ and $S(V)_1 = V$.

The character space $X(S(V))$ of $S(V)$ can be identified with the algebraic dual $V^*= \operatorname{Hom}(V,\mathbb{R})$ of $V$ by identifying $\alpha \in X(S(V))$ with $v^* \in V^*$ if $v^* = \alpha|_V$.
The topology on $V^*$ is the weak-* topology $\sigma(V^*, V)$, i.e.,\! the weakest topology such that $v^* \in V^*\mapsto v^*(f)\in \mathbb{R}$ is continuous $\forall f\in V$. 
Suppose now that $V$ is endowed with a locally convex topology $\tau$ defined by some (possibly uncountable) family of seminorms $S$ which can be always assumed to be directed. Then $V'$ is the (possibly uncountable) union of the subspaces $V_{\rho}'$, $\rho \in S$. In this setting Lemma \ref{lemma2} gives the following result about the uniqueness of the representing measure for a linear functional on $S(V)$ under the fundamental assumption that its support is contained only in countably many $V_{\rho}'$.
\begin{theorem}\label{at most one-general} 
Let $(V, \tau)$ be a separable l.c.\! topological space over $\RR$, where $\tau$ is defined by some (possibly uncountable) separating family $S$ of seminorms. Suppose that $\{ x_j : j \in \NN\}$ is a basis of a countable dimensional dense subspace $V_0$ of $V$ and that $L : S(V) \rightarrow \mathbb{R}$ is a linear functional such that the following condition holds for each $j \in \NN$: \begin{equation}\label{m}
\exists \text{ a sequence } \{p_{jk}\}_{k=1}^{\infty} \text{ in } S(V_0)\otimes \mathbb{C} \text{ such that }  \lim\limits_{k\rightarrow\infty}L(|1-(x_j-\mathrm{i})p_{jk}|^2)=0.
\end{equation}
Let $K:=\bigcup_{\rho\in T}V'_\rho$ for a countable subset $T$ of $S$. If there exists a $K-$representing measure for $L$ on $S(V)$, then this measure is unique.\end{theorem}

\begin{proof}
Let $T$ be a countable subset of $S$. By Lemma \ref{lemma2} applied to $V$ endowed with the l.c.\! topology generated by $T$, we have that the $\sigma$-algebra of Borel subsets of $(V_0)'$ coincides with the $\sigma$-algebra of Borel subsets of $\bigcup_{\rho\in T}V'_\rho$. Then the conclusion follows by applying \cite[Corollary 4.5]{GKM} to the $\mathbb{R}$-algebra $S(V_0) = \mathbb{R}[x_j : j \in \NN]$. Note that  the notion of constructibly Radon measure appearing in \cite[Corollary~4.5]{GKM} coincides in this case with the one of Radon measure
since $V_0$ is countable dimensional.
\end{proof}

Note that when $S$ is itself countable the previous theorem guarantees there is at most one $V'-$representing measure for $L$ on $S(V)$.

\section{The influence of the knowledge of the support in the determinacy question on $S(V)$}

In this section we further discuss Theorem \ref{at most one-general} collocating it in the framework of the determinacy of the moment problem for symmetric algebras of a l.c. space $V$ over $\RR$ (i.e., Problem \ref{MP} for $A=S(V)$; as mentioned in Section~\ref{SecDetCond} in this case $X(A)$ is identified with $V^*$). In \cite[Theorem 3.6]{I}, the author gives a general proof for a determinacy result which is commonly used for the moment problem for the symmetric algebras of a special class of l.c.\! nuclear spaces (we will introduce it later on). However, such a proof perfectly works also for the whole class of spaces which we have considered in Theorem \ref{at most one-general}. Indeed, the following holds.

\begin{theorem}\label{uniqueness-Suslin}
Let $(V,\tau)$ be a separable Hausdorff l.c.\! space over $\RR$. Suppose that $V'$ is a Suslin space and $L: S(V)\to \mathbb{R}$ is a linear functional such that $L$ is determining. Then there is at most one $V'-$representing measure for $L$ on $S(V)$.
\end{theorem}
Denote by $S$ a directed family of seminorms inducing the topology $\tau$ on $V$ and recall that:
\begin{itemize}
\item a Hausdorff topological space is said to be a \emph{Suslin space} if it is the image of a completely metrizable separable space under a continuous mapping.
\item a linear functional $L: S(V)\to \mathbb{R}$ is said to be \emph{determining} if:
\begin{enumerate}[(a)]
\item for each $k\ge 0$ the restriction map $L : S(V)_k \rightarrow \mathbb{R}$ is continuous w.r.t.\! the l.c.\! topology $\overline{\tau}_k$ on $S(V)_k$ induced by the seminorms $\{\overline{\rho}_k : \rho \in \mathcal{S}\}$,
where $\overline{\rho}_k$ denotes the quotient seminorm on $S(V)_k$ induced by $\rho^{\otimes k}$.
\item\label{d} there exists a countable subset $E$ of $V$ whose linear span is dense in $(V,\tau)$ such that, if
$$m_0 := \sqrt{L(1)}, \text{ and } m_d := \sqrt{\sup_{f_1, \dots ,f_{2d} \in E} |L(f_1\dots f_{2d})|}, \text{ for } d \ge 1,$$ then the class $C\{ m_k\}$
is quasi-analytic (see \cite[Definition 1.2]{I}).
\end{enumerate}
For more details about this condition, see \cite[Section 3 and 6]{GIKM} in particular \cite[Remark 6.2, (3)]{GIKM}. 
\item a linear functional $L: S(V)\to \mathbb{R}$ is called \emph{positive} if $L(\sum S(V)^2) \subseteq [0,\infty)$.
\end{itemize}

\begin{rem}\label{above}\ \\
Let $L$ be a {positive} linear functional on $S(V)$ and consider the following conditions.
\begin{enumerate}
\item $L$ determining.
\item There exists $\{x_j : j\in\NN\}$ countable subset of linearly independent vectors of $V$ whose linear span is dense in $(V,\tau)$ such that \eqref{d} holds.
\item The class $C\{\sqrt{|L(x_j^{2d})|}\}$ is quasi-analytic for each $j\in\NN$. 
\item The Carleman condition is fulfilled, i.e., $\sum\limits_{d=1}^\infty\frac{1}{\sqrt[2d]{|L(x_j^{2d})|}}=\infty$ for each $j\in\NN$.
\item The condition \eqref{m} holds for $V_0=span_{\RR}\{x_j : j\in\NN\}$. 
\end{enumerate}
\end{rem}
\noindent Then 1) $\Rightarrow$ 2) $\Rightarrow$ 3) $\Rightarrow$ 4) $\Rightarrow$ 5). It is easy to see that 2) $\Rightarrow$ 3) $\Rightarrow$ 4). For a proof of 1) $\Rightarrow$ 2) see \cite[Remark 6.2, (7)]{GIKM} and for a proof of 4) $\Rightarrow$ 5) see, e.g., \cite[Lemma~0.2 and Theorem~0.3]{M}. \\

Let us now compare Theorems \ref{at most one-general} and \ref{uniqueness-Suslin}. The assumptions on $L$ and $V'$ required in Theorem \ref{uniqueness-Suslin} are more restrictive than the ones in Theorem \ref{at most one-general}. However, Theorem \ref{at most one-general} gives a sufficient determinacy condition only in the case when we know the existence of representing measures supported by some countable union of $V'_\rho$ for $\rho\in S$. Theorem \ref{uniqueness-Suslin} gives instead a sufficient determinacy condition also for the general case when we have no further a priori information on the support of the solution to the moment problem except that it is contained in $V'$.

Note that if $S$ is countable we get that $V'$ is already Suslin. Indeed, for any $i\in\mathbb{N}$ and any $\rho\in S$ the ball $B_i(\rho')$ is a compact metric space (see Section 1) and so it is a Suslin space. Therefore, since each $V_{\rho}'$ is a countable increasing union of such balls, $V' = \bigcup_{\rho \in S} V_{\rho}'$ is a countable union of Suslin spaces and so it is Suslin itself. Hence, when $S$ is countable, Theorem \ref{uniqueness-Suslin} becomes a corollary of Theorem \ref{at most one-general} since the condition that $L$ is determining implies that \eqref{m} holds by Remark~\ref{above}.

As a last remark, we consider the special sort of l.c.\! spaces studied in \cite[Chapter 5, Section 2]{BK}, \cite{BS}, \cite[Section 3]{I} and \cite{IKR} for which existence results for the moment problem on $S(V)$ are known. Namely, $(V, \tau)$ is assumed to be: separable, the projective limit of a family $(H_j)_{j\in J}$ of Hilbert spaces ($J$ is an index set containing~$0$) which is directed by topological embedding such that each $H_j$ is embedded topologically into $H_0$, and nuclear, i.e.,\! for each $j_1\in J$ there exists $j_2\in J$ such that the embedding $H_{j_2}\hookrightarrow H_{j_1}$ is quasi-nuclear. Thus $\tau$ is the l.c.\! topology on $V$ induced by the directed family $S$ of the norms on $V$ which are induced by the embeddings $V \hookrightarrow H_j$, $j \in J$. The topology $\tau$ is usually referred to as the \it projective topology \rm on $V$ and it is clearly a Hausdorff topology.

\begin{theorem}[Berezansky-Kondratiev-\v Sifrin] \label{nuclear} \ \\
Let $(V,\tau)$ be a l.c.\! space of the special sort described above and let $L : S(V) \rightarrow \mathbb{R}$ be a linear functional. Assume that

(1) $L$ is positive;

(2) $L$ is determining.

\noindent Then there exists a $V'-$representing measure for $L$ on $S(V)$.
\end{theorem}
\begin{proof} See \cite[Chapter 5, Theorem 2.1]{BK} and \cite{BS}.
\end{proof}
The original proof of Theorem \ref{nuclear} actually shows, by use of the nuclearity assumption, a further fundamental property of the measure $\mu$ representing $L$: $\mu$ is indeed supported in a single $V'_{\rho}$ for some $\rho\in S$ (see \cite[Chapter~5, Section~2, Remark~1, p.72]{BK}). Then Theorem \ref{at most one-general} ensures that there exists a unique measure supported in $V'_{\rho}$ and representing $L$. (Note that under the assumptions of Theorem~\ref{nuclear} we can apply Theorem \ref{at most one-general} because of Remark \ref{above}). This means that thanks to the knowledge of the support of the representing measure provided by the proof of Theorem \ref{nuclear} we get determinacy on this special subset $V'_{\rho}$ of $V'$. This clearly does not guarantee that we have determinacy on the whole $V'$. Hence, to have $V'$-determinacy we need to verify that $V'$ is actually Suslin as required in Theorem~\ref{uniqueness-Suslin} (namely the Suslin assumption cannot be dropped in \cite[Theorem~3.7]{I}).

\section*{Acknowledgments}
We would like to thank the anonymous referee for her or his helpful comments and suggestions. We also express our gratitude to Mehdi Ghasemi and Tobias Kuna for the interesting discussions.

\end{document}